\documentclass[11pt,a4paper]{amsart}
\usepackage{a4wide,amsfonts,amsmath,latexsym,amssymb,euscript,graphicx,units,mathrsfs,bbold}   \usepackage[utf8]{inputenc}    
\usepackage[T1]{fontenc}
\usepackage{csquotes}
\usepackage{amsmath, amsthm, amsfonts, amssymb,color}
\usepackage{a4wide,amsfonts,amsmath,latexsym,amssymb,euscript,graphicx,units,mathrsfs,bbold}   

\usepackage{amsmath, amsthm, amsfonts, amssymb,color,comment}
\usepackage{bm}
\usepackage{theoremref}

\usepackage{enumitem}
\usepackage[margin=1in]{geometry}

\usepackage{mathtools}

\def\e{\epsilon}
\def\R{\mathbb R}
\def\E{\mathrm E}

\newtheorem{prop}{Proposition}[section]
\newtheorem{proposition}{Proposition}[section]

\newtheorem{lemma}[prop]{Lemma}
\newtheorem{definition}[prop]{Definition}
\newtheorem{corollary}[prop]{Corollary}
\newtheorem{theorem}[prop]{Theorem}

\numberwithin{equation}{section}
\begin{document}
   
   \title[Feynman-Kac formula for the  iterated derivatives of the PAM]{Feynman-Kac formula for iterated derivatives of  the parabolic Anderson model}

 \author[S.\ Kuzgun]{Sefika Kuzgun}
\address{University of Kansas, Department of Mathematics, USA}
\email{sefika.kuzgun@ku.edu}

\author[D.\ Nualart]{David Nualart} \thanks{Research supported  by the NSF grant DMS-1811181}
\address{University of Kansas, Department of Mathematics, USA}
\email{nualart@ku.edu}

\begin{abstract} The purpose of this paper is to establish a Feynman-Kac formula for the moments of the iterated Malliavin derivatives of the solution to the parabolic Anderson model in terms of  pinned Brownian motions. As an application,  we obtain estimates for the 
moments of the iterated derivatives of the solution.


\end{abstract}

\maketitle


\section{Introduction}

Consider the following stochastic heat equation (SHE) also referred as  the parabolic Anderson model (PAM):
\begin{equation}\label{PAM} 
  \frac{\partial u}{\partial t}= \frac{1}{2}\Delta u+ u \Dot{W},  \qquad x\in \R^d, \,\,  t \in (0,\infty), 
\end{equation}  
where $d\geq 1$, and $\Delta=\sum_{i=1}^d \frac{\partial ^2}{\partial x_i^2}$ is the Laplacian operator. We assume that $\Dot{W}$  is a centered Gaussian noise that is white in time and has an homogeneous covariance  in the space variable. This is to say, informally, \begin{align*}
    \E \left(\Dot{W}(t_1,x_1)\Dot{W}(t_2,x_2)\right)=\delta_0(t_1-t_2)\Lambda(x_1-x_2),
\end{align*}
where $\delta_0$ is the Dirac delta measure at zero, and $\Lambda$ is a non-negative definite tempered Borel measure on $\R^d$, satisfying Dalang's condition (\ref{Dalangscondition}). The initial condition $u_0$    is assumed to be a signed  Borel   measure on $\R^d$ such that for all $c>0$,
\begin{equation} \label{initial}
\displaystyle
 \int_{\R^d} e^{-c |x|^2} |u_0|(d x) <\infty.
\end{equation}
In particular,  the initial data  $u_0$ could be the Dirac measure $\delta_0$.

There is an extensive literature on the stochastic heat equation driven by an homogeneous noise. 
The existence and uniqueness of a mild solution to  a nonlinear equation (with the noise multiplied by a Lipschitz function $\sigma(u)$) was first proved in  \cite{dalang1999extending} when $u_0$ is a bounded function, and was extended to the case where $u_0$ satisfies (\ref{initial}) in \cite{ChenDalang13Heat} for the case where the noise is a space-time white noise in $d=1$, and in \cite{chen2019regularity} to the case where noise is white in time and satisfies Dalang's condition in space. 
Later the H\"older continuity in the space and time variables were established in \cite{Sanz-Sarra,ChenDalang14Holder}.  
Using techniques of Malliavin calculus, several authors studied  regularity  properties of the density of the solution; see, for instance \cite{CHN-19,chen2019regularity}.   When $\sigma(u)=u$, that means for equation  (\ref{PAM}), and $u_0$ is a bounded function,
Feynman-Kac formulas for the solution and its moments  were derived in \cite{hu2015stochastic}.

 Recently,  assuming $u_0=1$, the Malliavin-Stein methodology has been applied to  establish ergodicity in the space variable and to derive quantitative central limit theorems for spatial averages, see \cite{CKNP-1,CKNP-2,HuangNualartViitasaari2018,HuangNualartViitasaariZheng2019}.
  A basic ingredient for these results is the fact that the $p$-norm of the Malliavin derivative of the solution at a given point, $\| D_{s,y} u(t,x)\|_p$ is bounded by a constant times the fundamental solution 
 $p_{t-s}(x-y)$ to the heat equation.  Motivated by these applications,  the aim of this paper is to establish a  Feynman-Kac formula for  the moments of the iterated derivatives of the solution $u(t,x)$ to  equation (\ref{PAM}) and to derive moment estimates.

Let us introduce some notation. For any integer $N\ge 1$,  $D^N_{\pmb{r}_N,\pmb{z}_N}u(t,x)$ denotes the $N$-th iterated derivative of $u(t,x)$ in the sense of Malliavin calculus, that is
 \begin{align*}
    D^N_{\pmb{r}_N,\pmb{z}_N}u(t,x)=D_{r_1,z_1}\cdots D_{r_N,z_N}u(t,x),
\end{align*}
for any  $z_1, \dots, z_N \in \R^d$ and $0<r_1 < \cdots < r_N$ and with the notation $\pmb{r}_N=(r_1, \dots, r_N)$ and $\pmb{z}_N=(z_1, \dots, z_N)$.  We will denote by $p_{t}(x)$ the $d$-dimensional heat kernel:
\begin{equation} \label{p}
p_{t}(x)= (2\pi t)^{-d/2} \exp( -|x|^2/(2t)), \qquad x\in \R^d, \,\, t>0.
\end{equation}

Our main result, stated below, provides an explicit formula for the moments of order $k$ of the iterated Malliavin derivatives of $u(t,x)$ in terms of the expectation of an exponential functional of a $k$-dimensional pinned Brownian motion.

\begin{theorem} \label{thm1}
Let $u$ be the unique solution of  (\ref{PAM}) with initial condition $u_0$ which is a signed Borel measure satisfying (\ref{initial}). Let $N \ge 1$   be  an integer and $z_1,\dots, z_N \in \mathbb{R}^d$, $0<r_1<\cdots <r_N <t$. Then for  integer $k\ge 2$, we have
\begin{align*}
    \E \left[ \left( D^N_{\pmb{r}_N,\pmb{z}_N}u(t,x)\right)^k \right] &=  \left[  \prod_{m=1}^{N-1} p_{r_{m+1}-r_m}(z_{m+1}-z_m) \right]^k p^k_{t-r_N}(x-z_N) \\
 & \qquad \times   \int_{\R^{kd}}  \prod_{j=1}^k u_0(d\theta^j)  \prod_{j=1}^kp_{r_1}(z_1- \theta^j)  \\
 & \qquad \times \E  \left(\exp{\left(\sum_{1\leq j<l\leq k}\int_0^t  \Lambda(\widehat{B}_{0,\pmb{t-r}_N,t}^{j,x, \pmb{z}_N, \theta^j}(s)-\widehat{B}_{0,\pmb{t-r}_N,t}^{l,x,\pmb{z}_N, \theta^l}(s))ds \right)}  \right),
\end{align*}
 where   $\pmb{t-r}_N=(t-r_1,\dots,t-r_N)$ and $\widehat{B}_{0,\pmb{t-r}_N,t}^{j,x,\pmb{z}_N, \theta^j}$, $j=1,\dots, k$ are independent  $d$-dimensional pinned Brownian motions starting from $x$  with each component pinned at times $t-r_m$ to  the points $z_m$ for $1\leq m\leq N$, and pinned at  $\theta^j$ at time $t$.
 
 Moreover,
 \[
\E  \left(\exp{\left(\sum_{1\leq j<l\leq k}\int_0^t  \Lambda(\widehat{B}_{0,\pmb{t-r}_N,t}^{j,x, \pmb{z}_N, \theta^j}(s)-\widehat{B}_{0,\pmb{t-r}_N,t}^{l,x,\pmb{z}_N, \theta^l}(s))ds \right)}  \right)\le C_{t,k},
  \]
  where $C_{t,k}$ is a constant depending only on $t$ and $k$.
  \end{theorem}
  
  In the above theorem, taking into account that $\Lambda$ might be a measure,  the composition  $\Lambda(\widehat{B}_{0,\pmb{t-r}_N,t}^{j,x, \pmb{z}_N, \theta^j}(s)-\widehat{B}_{0,\pmb{t-r}_N,t}^{l,x,\pmb{z}_N, \theta^l}(s))$ needs to be  properly defined  as a limit in $L^2(\Omega)$, using an approximation argument, see   Proposition \ref{basic}, part (ii). When $\Lambda(dx) = \Lambda(x) dx$, then this is just an ordinary composition of the density $\Lambda$ with the random variable  $\widehat{B}_{0,\pmb{t-r}_N,t}^{j,x, \pmb{z}_N, \theta^j}(s)-\widehat{B}_{0,\pmb{t-r}_N,t}^{l,x,\pmb{z}_N, \theta^l}(s)$. 
  
  As a consequence of Theorem \ref{thm1},  we deduce the following result.
  
  \begin{corollary} \label{cor1}
 Under the assumptions and notation of Theorem  \ref{thm1}, we have
 \begin{equation} \label{ec1}
 \left\|  D^N_{\pmb{r}_N,\pmb{z}_N}u(t,x)\right \|_k  \leq C^{1/k} _{t,k}  (p_{r_1}*|u_0|)(z_1)\left(\prod_{m=1}^{N-1}  p_{r_{m+1}-r_m}(z_{m+1}-z_m)\right) p_{t-r_N} (x-z_N).
\end{equation}
\end{corollary}

Here is a survey of previous results related to Corollary \ref{cor1}, obtained by alternative methods.
The  estimate (\ref{ec1}) for $N=1$ was first proved  in  \cite[Lemma 5.1]{HuangNualartViitasaari2018} for the nonlinear stochastic heat equation with $u_0=1$, driven by a space-time white noise,   and later extended to the case of a Riesz-type covariance in  \cite[Lemma 2.1]{HuangNualartViitasaariZheng2019} and to  the case of a general spatial covariance satisfying  (\ref{Dalangscondition}) in     \cite[Theorem 4.6]{CKNP-1}. In the paper \cite{CKNP-4}, the authors have also considered the case $N=2$, but only for the parabolic Anderson  model (\ref{ec1}), the case of a general coefficient $\sigma$ being an open problem. Finally, in the case of equation (\ref{PAM}) with $u_0=\delta_0$, the estimate  (\ref{ec1}) for $N=1$  has been obtained in \cite[Lemma 2.1]{CKNP-3}. 

The paper is organized as follows. After some preliminaries in Section 2, Section 3 is devoted to establishing a Feynman-Kac formula for the solution to equation 
(\ref{PAM}) driven by a regularization $W^\epsilon$ of the noise. In this section we also derive a Feynman-Kac formula for the moments of the solution to equation
(\ref{PAM}).  Finally, Section 4 contains the proof of  Theorem \ref{thm1}.


\section{Preliminaries}

\subsection{The Set-Up}
Firstly, we introduce some notation.  The space of Schwartz functions, that is,  the space of infinitely differentiable rapidly decreasing functions on $\R_+\times \R^d$, respectively on $ \R^d$,  will be denoted by $\mathscr{S}(\R_+\times \R^d)$, respectively by $\mathscr{S}(\R^d)$.  Then $\mathscr{S}'(\R^d)$ will denote the dual of $\mathscr{S}(\R^d)$ and its elements are the so-called tempered distributions. The Fourier transform of a tempered Borel measure $\nu$ on $\R^d$ is defined by \begin{align*}
    \mathscr{F}(\nu)(\xi)=\int_{\R^d}e^{- i \xi\cdot x}\nu(dx),
\end{align*}
which is a slowly increasing $C^{\infty}$ function.  

Let $(\Omega, \mathcal{F}, P)$ be a complete probability space, and $W$ a Gaussian noise encoded by a centered Gaussian family $\{W(\varphi); \varphi \in \mathscr{S}(\R_+\times \mathbb{R}^d)\}$ with the covariance structure \begin{align} \label{CovarianceSturucture}
    E\left(W(\varphi)W(\psi)\right)=\big\langle \varphi, \psi * \big( \delta_0 \times \Lambda   \big)\big\rangle_{L^2(\R_+\times \R^d)} ,
\end{align}
where   $\Lambda$ is a nonnegative definite tempered Borel measure on $\R^d$.  We assume that the  Fourier transform of $\Lambda$, denoted  by $ \mathscr{F}\Lambda=\mu$, is a measure satisfying Dalang's condition:
\begin{align} \label{Dalangscondition}
    \int_{\mathbb{R}^d} \frac{\mu(d\xi)}{1+|\xi|^2} <\infty.
\end{align}  

Let $\mathcal{H}$ be the completion of $\mathscr{S}(\R_+\times \mathbb{R}^d)$ endowed with the inner product 
\begin{align} \label{innerproduct}
    \langle \varphi , \psi\rangle_{\mathcal{H}}&=\big\langle \varphi, \psi * \big( \delta_0 \times \Lambda  \big)\big\rangle_{L^2(\R_+\times \R^d)}\\  \notag
    &=\int_0^{\infty} \int_{ \mathbb{R}^{2d}} \varphi{(s,y)}\psi(s,y-y')\Lambda(dy')dyds \\
    \nonumber &= \int_0^{\infty} \int_{ \mathbb{R}^{d}} \mathscr{F}\varphi(s,\cdot)(\xi) \overline{\mathscr{F}\psi(s, \cdot)(\xi)}\mu(d\xi)ds.
\end{align}

The mapping $\varphi \mapsto W(\varphi)$ defined on $\mathscr{S}(\R_+\times \mathbb{R}^d)$  can be extended to a linear isometry between $\mathcal{H}$ and the Gaussian space spanned by $W$. This isometry will be denoted by \begin{align}
    W(\varphi)=\int_0^{\infty} \int_{ \mathbb{R}^{d}} \varphi(s,y)W(ds,dy)
\end{align}
for $\varphi \in \mathcal{H}$. With this notation in mind, we have $\E\left(W(\varphi) W( \psi ) \right)=\langle \varphi , \psi \rangle_{\mathcal{H}}$. We will use the notation $W(t,x)$ as shorthand for $W(\mathbb{1}_{[0,t]\times[0,x]})$, so that we can write \begin{align*}
    \E\left(W(t,x) W(t',x' ) \right)&=\left(t\wedge t'\right) \int_{\mathbb{R}^{2d}}\mathbb{1}_{[0,x]}(y)\mathbb{1}_{[0,x']}(y-y') \Lambda(dy')dy \\
    &= (t\wedge t')\int_{\mathbb{R}^d} \mathscr{F}\mathbb{1}_{[0,x]}(\cdot)(\xi)\overline{\mathscr{F}\mathbb{1}_{[0,x']}(\cdot)(\xi)}\mu(d\xi).
\end{align*}
Let $\mathcal{H}_0$ be  the completion of $\mathscr{S}(\mathbb{R}^d)$ with the inner product 
\begin{align}
    \langle \phi_1 ,\phi_2 \rangle_{\mathcal{H}_0} =\int_{\mathbb{R}^{2d}} \phi_1(y)\phi_2(y-y')\Lambda(dy')dy.
\end{align}
In this way, we have  $\mathcal{H} = L^2( \R_+; \mathcal{H}_0)$.

For each $t \geq 0$, let $\mathcal{F}_t$ be the $\sigma$-field generated by the random variables $W(\varphi)$ where $\varphi$ has support in $[0,t]\times \mathbb{R}^d$.  We say that a random field $u=\{u(t,x); (t,x) \in [0,\infty) \times \mathbb{R}^d\}$
 is adapted if  for each $(t,x)\in [0,\infty)\times \R^d$ the random variable $u(t,x)$ is $\mathcal{F}_t$ measurable.

\begin{definition}{\label{solution}}  
 A random field $u=\{u(t,x); (t,x) \in [0,\infty) \times \mathbb{R}^d\}$ is called a  mild solution to (\ref{PAM})  with initial condition $u_0$, if
$u$ is adapted,  jointly measurable with respect to $\mathcal{B}\left([0,\infty)\times \R^d\right) \times \mathcal{F}$,
$\E\left(u(t,x)^2 \right) <\infty$ for all $(t,x)\in (
0,\infty)\times \mathbb{R}^d$  and $u$ satisfies the integral equation
 \begin{equation}\label{integralform}
    u(t,x)= (p_t * u_0)(x)+  \int_{0}^t \int_{\R^d} p_{t-s}(x-y)u(s,y)W(ds,dy), \qquad {\rm a.s.}
\end{equation}
for all $(t,x)\in 0,\infty)\times \R^d$,
 where the above stochastic integral is to be understood in the sense defined by Walsh in \cite{walsh1986introduction} and extended by Dalang in \cite{dalang1999extending}.
\end{definition}


\subsection{Malliavin Calculus}

In this subsection we will  recall some basic elements of the Malliavin calculus associated with $W$. 
For more details on the Malliavin calculus, we refer to \cite{eulalia,nualart2006malliavin}.
For a smooth and cylindrical random variable $F= f(W(\varphi_1), \dots , W(\varphi_n))$, with $\varphi_i \in  \mathcal{H}$, $1\le i \le n$ and $f \in C_b^{\infty}(\mathbb{R}^n)$ ($f$ and all  its partial derivatives are bounded), we define its Malliavin derivative as an $\mathcal{H}$-valued
random variable  defined as
\[
 DF = \sum_{i=1}^n \frac{\partial f}{\partial x_i} (W(\varphi_1), \dots, W(\varphi_n))\varphi_i\ .
\]
By iteration, we can also define the $k$-th derivative $D^k F$, which is an element in the space $L^2(\Omega; \mathcal{H}^{\otimes k})$. For any real $p\geq 1$ and any integer $k\geq 1$, the Sobolev space $\mathbb{D}^{k,p}$ is defined as the closure of the space of smooth and cylindrical random variables with respect to the norm $\|\cdot\|_{k,p}$ defined by 
\[
 \|F\|^p_{k,p} = \mathbb{E}(|F|^p) + \sum_{i=1}^k \mathbb{E}(\|D^i F\|^p_{\mathcal{H}^{\otimes i}}).
\]
We set $\mathbb{D}^{\infty}:=\displaystyle\cap_{k,p \in \mathbb{N}}  \mathbb{D}^{k,p}$.

The following result  ensures on the regularity of the solution to equation (\ref{PAM}) in the sense of Malliavin calculus.

\begin{proposition} For any $(t,x) \in (0,\infty)\times \R^d$, $u(t,x)\in \mathbb{D}^{\infty}$.
\end{proposition}

\begin{proof}
From Part (2) of  \cite[Proposion 3.2]{chen2019regularity} it follows that  $u(t,x)\in \mathbb{D}^{1,p}$ for all $(t,x) \in (0,\infty)\times \R^d$ and for all $p\ge 1$.
Because we are dealing with the parabolic Anderson model, the proof of Part (3) of  \cite[Proposion 3.2]{chen2019regularity} implies that $u(t,x)\in \mathbb{D}^{\infty}$ for all $(t,x) \in (0,\infty)\times \R^d$.
\end{proof}

\subsection{Brownian bridges}

Along the paper, $\{\widehat{B}_{[a,b]} ^{x,y}(s); s\in [a,b]\}$ will denote a  $d$-dimensional Brownian bridge in the time interval $[a,b]$ that goes from the starting point $x$ at time $a$  to the end point $y$ at time $b$.  We also set  $\widehat{B}_{[a,b]} := \widehat{B}_{[a,b]} ^{0,0}$.
We recall that  the Brownian  bridge $\widehat{B}_{[a,b]} ^{x,y}$ can be expressed as
\begin{equation} \label{BB}
\widehat{B}_{[a,b]} ^{x,y}(s)= \widehat{B}_{[a,b]} (s)+ \frac {s-a}{b-a} y + \frac {b-s}{b-a} x, \quad x,y \in \R^d.
\end{equation}

\section{Regularization of the Noise}
We will introduce the following  regularization of the noise $W$ in the space variable. For each $\epsilon>0$ and any 
$\varphi \in \mathscr{S}(\R_+\times \R^d)$, we define
\[
W^{\epsilon}(\varphi)= W(\varphi * p_{\epsilon}),
\]
where $*$ denotes the convolution in the space variable  and $p_{\epsilon}(x) $  is the $d$-dimensional heat kernel  defined in (\ref{p}).
Then, the Gaussian family 
$W^{\epsilon}=\left\{W^{\epsilon}(\varphi); \varphi \in \mathscr{S}\left(\R_+\times \R^d\right)\right\}$ has the covariance structure  
\begin{align*}
    \E\left(W^{\epsilon}(\varphi)W^{\epsilon}(\psi)\right)&=\int_{0}^{\infty}\int_{\mathbb{R}^{2d}} \left(\varphi(s,\cdot)*p_{\epsilon}(\cdot)\right)(y)\left(\psi(s,\cdot)*p_{\epsilon}(\cdot)\right)(y-y')\Lambda(dy')dyds \\ 
    &= \int_{0}^{\infty}\int_{\mathbb{R}^d} \mathscr{F}(\varphi)(s,\xi)\overline{\mathscr{F}(\psi)(s, \xi)}e^{-\epsilon |\xi|^2}\mu(d\xi)ds,
\end{align*}
that is, the  noise $W^\epsilon$ is white in time and it has a spatial covariance given by
\begin{equation}\label{lambdaepsilon}
 \Lambda_{\epsilon}(x)= \frac 1{ (2\pi)^d}\int_{\mathbb{R}^d} e^{ix\cdot \xi-\epsilon|\xi|^2}\mu(d\xi), 
 \end{equation}
whose Fourier transform is
$ \mu_{\epsilon}(d\xi)=  e^{-\epsilon|\xi|^2}\mu(d\xi) $.
Notice that  $ \mu_{\epsilon}$ is a finite measure and  $ \Lambda_{\epsilon}$ is a bounded smooth function.
In this way, we can write \begin{align*}
     \E\left(W^{\epsilon}(\varphi)W^{\epsilon}(\psi)\right)&=\int_0^{\infty}\int_{ \mathbb{R}^{2d}} \varphi(s,y)\psi(s,y')\Lambda_{\epsilon}(y-y')dydy'ds \\ &= \int_0^{\infty}\int_{ \mathbb{R}^{d}} \mathscr{F}(\varphi)(s,\xi)\overline{\mathscr{F}(\psi)(s, \xi)}\mu_{\epsilon}(d\xi)ds.
\end{align*}
As before, we denote by $\mathcal{H}^\epsilon$ the completion of  $\mathscr{S}(\R_+\times \mathbb{R}^d)$ under the inner product $$\langle \varphi, \psi \rangle_{\mathcal{H}^\epsilon}=   \E\left(W^{\epsilon}(\varphi)W^{\epsilon}(\psi)\right).$$
\vskip 10pt

For a fixed $t \in \R_{\geq 0}$, let $f^t:[0,t]\to \R^d $ be a continuous function. Then, the map $(s,y) \mapsto  \mathbb{1}_{[0,t]}(s)p_{\epsilon}\left(f^t(s)-y\right)$ belongs to the space $\mathcal{H}$ since
 \begin{align}  \notag
\|\mathbb{1}_{[0,t]}(\bullet)p_{\epsilon}\left(f^t(\bullet)-\star\right)\|^2_{\mathcal{H}} &= \int_0^{t} \int_{\R^{2d}} p_{\epsilon}\left(f^t(s)-y\right)p_{\epsilon}\left(f^t(s)-y'+y\right)\Lambda(dy')dyds \\ \notag
&= 
\int_{0}^{t} \int_{\R^d}e^{-\epsilon|\xi|^2}\mu(d\xi)ds=t \int_{\R^d}e^{-\epsilon|\xi|^2}\mu(d\xi) \\
&=t(2\pi)^d \Lambda _\epsilon(0)  <\infty   \label{lambdazero}
\end{align}
and we can define the stochastic integral 
\[
 W\left( \mathbb{1}_{[0,t]}(\bullet)p_{\epsilon}\left(f^t(\bullet)-\star\right)\right)
 =\int_0^t \int_{\mathbb{R}^d} p_{\epsilon}(f^t(s)-y)W(ds,dy).
\] 
 Throughout, we will use the following notation: 
\[
\int_0^t W^{\epsilon}(ds,f^t(s)):=\int_0^t \int_{\mathbb{R}^d} p_{\epsilon}(f^t(s)-y)W(ds,dy).
\]
From (\ref{lambdazero}) it follows that  $\int_0^t W^{\epsilon}(ds,f^t(s))$ is a centered Gaussian random variable with variance  $t(2\pi)^d \Lambda _\epsilon(0)$.

The following result will play an important role in  the proof of our main result.

\begin{proposition} \label{basic}
Fix an integer $k\ge 2$.
Let   $B^j_{[a,b]}$, $j=1, \dots, k$ be  independent   $d$-dimensional Brownian bridges in $[a,b]$ from $0$ to $0$, where $[a,b] \subset [0,t]$.  Consider  a
measurable  function 
$\alpha= \left(\alpha^{j,l}\right)_{1\le j < l \le k}: [a,b] \to \R^{k(k-1)/2}$. For each $1\le j<l \le k$ we set
\[
G^{j,l}_\e:=\int_a^b \Lambda_\e(B^j_{[a,b]}(s)- B^l_{[a,b]}(s) + \alpha^{j,l} (s)) ds.
\]
Then the following results hold true:

\medskip
\noindent
{\it (i)} For each $\kappa \in \R$,
\begin{equation} \label{kappa1}
\sup_{\e \in (0,1]}  \sup_\alpha \E \left( \exp \left( \kappa \sum_{1\le j < l \le k} G^{j,l}_\e \right) \right)
 = K_{t,\kappa} <\infty,
\end{equation}
where the constant   $K_{t,\kappa} $ only depends on $t$ and $\kappa$.

 \medskip
\noindent
{\it (ii)}  For each $1\le j<l \le k$, the random variables   $
G^{j,l}_\e$
  converge   in $L^p(\Omega)$ for all $p\ge 2$, as $\e \downarrow 0$,  to a limit denoted by  $G^{j,l}:=\int_a^b \Lambda(B^j_{[a,b]}(s)- B^l_{[a,b]}(s) + \alpha^{j,l} (s)) ds$.

\medskip
\noindent
{\it (iii)} We have, for all $\kappa \in \R$,
\[
  \lim _{ \e \downarrow 0}  \E \left( \exp \left( \kappa   \sum_{1\le j < l \le k}   G^{j,l}_\e  \right)\right)
=
  \E \left( \exp \left( \kappa   \sum_{1\le j < l \le k}  G^{j,l} \right)\right),
\]
where the convergence is uniform in $\alpha$ and in $a,b$.
\end{proposition} 

\begin{proof}
Property  (\ref{kappa1}) has been proved in \cite[Lemma 4.1]{huang2017large}. The convergence in  point (iii), as $\e $ tends to zero, follows from \cite[Proposition 4.2]{huang2017large}. Actually, in Proposition 4.2 of \cite{huang2017large}, the result is proved for  $\alpha^{j,l} (s)= x^j-x^l$, where $x^j \in \R^d$, $j=1,\dots, k$, but the case of a general function $\alpha$ can be done in the same way. Property (ii) is proved in  Proposition 4.3 of \cite{huang2017large} for Brownian motions and with $\alpha^{j,l} (s)= x^j-x^l$ and  the arguments  of the proof   are  also valid for   Brownian  bridges   and for a general function $\alpha$.
\end{proof}

Now, we consider the heat equation driven by $W^{\epsilon}$,
\begin{equation}\label{SHEepsilon}
  \frac{\partial u}{\partial t}= \frac{1}{2}\Delta u+ u \Dot{W}^{\epsilon}, \qquad x\in \R^d,  \,\, t \in \R_+,
\end{equation}
with initial condition $u(0,x)=u_0$ where $u_0$  is a signed Borel measure satisfying (\ref{initial}). 
 An adapted  and jointly measurable random field $u^{\epsilon}=\{u^{\epsilon}(t,x); (t,x) \in (0,\infty) \times \mathbb{R}^d\}$ such that $\E\left(u^{\epsilon}(t,x) \right)^2 <\infty$ for all $(t,x)\in (0,\infty)\times \mathbb{R}^d$ is a  mild solution to  equation (\ref{SHEepsilon}), if for any $(t,x)\in (0,\infty)\times \mathbb{R}^d$, the process $\{p_{t-s}(x-y)u^{\epsilon}(s,y)\mathbb{1}_{[0,t]}(s); (s,y)\in (0,\infty)\times \mathbb{R}^d\}$ is integrable with respect to $W^\epsilon$, and the following holds:
\begin{equation} \label{ecu1}
    u^{\epsilon}(t,x)=(p_t * u_0)(x)+\int_{0}^t\int_{ \mathbb{R}^d} p_{t-s}(x-y)u^{\epsilon}(s,y) W^{\epsilon}(ds,dy).
\end{equation}
It follows from the general theory that this mild solution exists and it is unique. Furthermore, because the spectral measure is finite, there is a Feynman-Kac representation of the solution, given in the following lemma. For the sake of completeness we include a proof of the lemma, based on a duality argument.

\begin{lemma} \label{lem1} For each $\epsilon>0$, the following random field $u^{\epsilon}(t,x)$ is the solution to the heat equation given in (\ref{SHEepsilon}):
\begin{align}\label{epsilonsolution}
    u^{\epsilon}(t,x)=\E^{B}\left( u_0(B^x_t)\exp \left(\int_0^t W^{\epsilon}\left(ds,B_{t-s}^x\right)-\frac{1}{2}t (2\pi)^d\Lambda_{\epsilon}(0)\right)\right),
\end{align}
where $B^x$ is a $d$-dimensional standard Brownian motion independent of $W$ that starts at $x$ and  $\E^B$ denotes the mathematical expectation with respect to $B^x$.
\end{lemma}

\medskip
\noindent
 {\bf Remark 1.} Notice that,  because $u_0$ is a signed measure, the composition $u_0(B^x_t)$ is not well defined. The right-hand side of equation
 (\ref{epsilonsolution}),  can be interpreted in two ways:
 
 \smallskip
 \noindent
 (i) We can write 
 \[
  u^{\epsilon}(t,x)=  \int_{\R^d}  u_0 (d\theta) p_t(x-\theta)) \E^{\widehat{B} }\left(  \exp \left(\int_0^t   \int_{\R^d} 
  p_{\e} (\widehat{B}^{\theta,x}_{0,t} (s)-y ) W(ds,dy)  -\frac 12 t(2\pi)^d\Lambda_{\epsilon}(0)\right)\right),
  \]
  where $\{\widehat{B}^{\theta,x}_{0,t}(s), s\in [0,t]\}$ denotes a $d$-dimensional Brownian bridge in the interval $[0,t]$ from $\theta$ to $x$.
 The above integral is well defined almost surely because on one hand $\int_{\R^d}  |u_0| (d\theta) p_t(x-\theta)) <\infty$ and moreover,  from the computations in (\ref{Phi}), we have
 \[
  \E^W \E^{\widehat{B} }\left(  \exp \left(\int_0^t   \int_{\R^d} 
  p_{\e} (\widehat{B}^{\theta,x}_{0,t} (s)-y ) W(ds,dy) \right)\right) = e^{ \frac 12 t(2\pi)^d\Lambda_{\epsilon}(0)} .
  \]
  
  \smallskip
 \noindent
 (ii)  From the results in  \cite{CHN-19}  the random variable   $u_0(B^x_t)$ belongs to Meyer-Watanabe space $\mathbb{D}^{-\alpha, p}$ for  any $p>1$ and 
 $\alpha>1- \frac 1p$. Furthermore,  it can be proved that, conditionally to $W$,   the exponential term 
 $\mathcal{E}:=\exp \left(\int_0^t   \int_{\R^d}   p_{\e} (B_{t-s}^x-y ) W(ds,dy)\right)$ is in the space $\mathbb{D}^{1,2}$ as a functional of the Brownian motion $B$, with a derivative given by
 \[
 D_r \mathcal{E} =\mathcal{E}  \sum_{i=1}^d  \int_0^{t-r}   \int_{\R^d}   \frac {\partial p_\e} {\partial x_i} (B_{t-s}^x-y ) W(ds,dy).
 \]
 Then, the right-hand side of equation
 (\ref{epsilonsolution}) can be expressed as the following pairing
 \[
e^{ -\frac 12\Lambda_{\epsilon}(0)}
 \left \langle u_0(B^x_t), \mathcal{E} \right \rangle_{\mathbb{D}^{-1,2}, \mathbb{D}^{1,2} }.
 \]

\begin{proof}[Proof of Lemma \ref{lem1}]
Let $G \in L^2(\Omega, \mathcal{F}, P)$ be such that $G=e^{W(h)-\frac{1}{2}\|h\|_{\mathcal{H}}^2}$ for some $h\in \mathcal{H}$.  From (\ref{epsilonsolution}), we obtain
\begin{align}
   \nonumber \E\left( Gu^{\epsilon}(t,x)\right)&= \E^W\left( GE^{B}\left( u_0(B^x_t) \exp{\left(\int_0^t W^{\epsilon}\left(ds,B_{t-s}^x\right)-\frac{1}{2}t(2\pi)^d\Lambda_{\epsilon}(0)\right)}\right)\right) \\
    \nonumber &=\E^B \left(u_0(B^x_t)\E^W \left(\exp{\left( W(h+p_{\epsilon}(B_{t-\bullet}^x-\star)) -\frac{1}{2}\|h\|^2_{\mathcal{H}}-\frac{1}{2}t(2\pi)^d \Lambda_{\epsilon}{(0)}\right)}\right)\right)\\ \nonumber 
    & =\E^B \left(u_0(B^x_t)\exp{\left(\frac{1}{2}\|h+p_{\epsilon}(B_{t-\bullet}^x-\star)\|^2_{\mathcal{H}} -\frac{1}{2}\|h\|^2_{\mathcal{H}}-\frac{1}{2}t(2\pi)^d\Lambda_{\epsilon}{(0)}\right)}\right)  \\ \nonumber 
    &= \E^B \left(u_0(B^x_t)\exp{\left( \big\langle p_{\epsilon}(B_{t-\bullet}^x-\star) , h \big\rangle^2_{\mathcal{H}}\right)}\right) \\ \nonumber 
    &= \E^B \left(u_0(B^x_t)\exp{\left( \int_0^t \big\langle p_{\epsilon}(B_{t-s}^x-\star) , h(s,\star) \big\rangle_{\mathcal{H}_0}ds\right)}\right) .
\end{align}
Letting $S_{t,x}(h)=\E^W(Gu^{\epsilon}(t,x))$,   by the classical Feynmann-Kac's formula, the above calculation shows that $S_{t,x}(h)$ satisfies the classical heat equation with   potential $\langle p_{\epsilon}(x-\star) , h(s,\star) \big\rangle_{\mathcal{H}_0}$,  and initial condition $u_0$, i.e.
\begin{align}
    \nonumber \frac{\partial S_{t,x}(h)}{\partial t} =\frac{1}{2}\Delta S_{t,x}(h)+S_{t,x}(h)\langle p_{\epsilon}(x-\star) , h (t,\star) \big\rangle_{\mathcal{H}_0}.
\end{align}
 As a consequence, we have
  \begin{align}
    \nonumber S_{t,x}(h)&=(p_t *u_0)(x)+ \int_0^{t}\int_{ \mathbb{R}^{d}}p_{t-s}(x-y)S_{s,y}(h)\langle p_{\epsilon}(y-\star) , h (s,\star)\rangle_{\mathcal{H}_0}dsdy \\ \nonumber 
    &=(p_t *u_0)(x)+ \int_0^{t}\int_{ \mathbb{R}^{d}}p_{t-s}(x-y)\E\left(u^{\epsilon}_{s,y}\langle p_{\epsilon}(y-\star) ,D_{s,\star}G \rangle_{\mathcal{H}_0}\right)dsdy,
\end{align}
where we used $DG=hG$.  In conclusion, we have proved that
\[
    \E\left( Gu^{\epsilon}(t,x)\right)= (p_t *u_0)(x)+\E \left(\left\langle  \mathbb{1}_{[0,t]} (\bullet) \int_{ \mathbb{R}^d} p_{t-\bullet}(x-y)p_{\epsilon}(y-\star)u^\e(\bullet,y) dy, DG\right\rangle_{\mathcal{H}}\right).
\]
By the fact that the Dalang-Walsh stochastic integral is the adjoint of the Malliavin derivative, we deduce that
\[
    u^{\epsilon}(t,x)=(p_t *u_0)(x)+\int_{0}^t\int_{ \mathbb{R}^d}  \left( \int_{\R^d} p_{t-s}(x-y) p_\epsilon(y-z)u^{\epsilon}(s,y)  dy  \right)W(ds,dz),
\]
which implies equation  (\ref{ecu1}).
\end{proof}

In the next theorem we show that $u^\epsilon(t,x)$ converges to the solution $u(t,x)$ of the stochastic heat equation (\ref{PAM}) in $L^p(\Omega)$ for all $p\ge 1$,  and, as a consequence, we derive a Feynman-Kac formula for the moments of the solution. This type of Feynman-Kac formula  has been established in the  literature under different conditions (see, for instance, \cite[Theorem 3.6]{hu2015stochastic} for the case where $\Lambda$ is a function and there is also a correlation in time, or \cite{HN}      when the noise is white in space and a fractional Brownian motion with Hurst parameter $H>1/2$ in time)
assuming that $u_0$ is a bounded function. We will give here a detailed proof based on the approximation scheme $u^\epsilon(t,x)$, because the necessary computations  will be also used in the proof of Theorem \ref{thm1}.

\begin{proposition} \label{FeynmanKacTheorem}  Let $u^\epsilon$ be the solution to equation (\ref{ecu1}) with an initial condition $u_0$ satisfying (\ref{initial}). Then, for any $k \geq 1$, we have
 \begin{equation}  \label{ecu3}
    \sup_{\epsilon>0} \E \left( \left| u^{\epsilon}(t,x)\right|^k\right) <\infty
\end{equation}
and the following convergence holds in $L^p(\Omega)$ for any $p \geq 1$:
\begin{align}
    \lim_{\epsilon \to 0} u^{\epsilon}(t,x)=u(t,x),
\end{align} 
where  $u$ is the solution to the stochastic heat equation (\ref{PAM}) with initial condition $u_0$.  Furthermore, for any integer $k \geq 2$, the following Feynmann-Kac formula holds: 
\begin{align} \label{FeynmanKac}
\E\left[u^k(t,x) \right]=\E\left( \prod_{j=1} ^k u_0(B^{j,x}_t) \exp{\left( \sum\limits_{1\leq j <l\leq k}\int_0^t \Lambda(B_{s}^j- B_{s}^l)ds\right)} \right),
\end{align}
where $B=\{B^j\}_{j=1,\dots ,k}$ is an independent  family of $d$-dimensional standard Brownian motions and the integrals 
$\displaystyle\int_0^t \Lambda(B_s^{j}- B_s^l)ds$ are defined  according to
 Proposition \ref{basic} (ii). 
\end{proposition}

\begin{proof}
Set   $\Psi^k_{t,x} =:\prod_{j=1} ^k u_0(B^{j,x}_t)$.
Using Lemma \ref{lem1}, we have
 \begin{align}\label{momentcalculation}\
        \E\left[\left(u^{\epsilon}(t,x)\right)^k\right] =\E^{W}\E^B \left(\Psi^k_{t,x} \exp \left(\sum\limits_{j=1}^k \int_0^t W^{\epsilon} (ds, B^{j,x}_{t-s})-\frac{1}{2}t(2\pi)^d\Lambda_{\epsilon}(0)\right)\right),
    \end{align}
    where $B=\{B^j\}_{j=1,\dots ,k}$ is a family of $d$-dimensional independent standard Brownian motions independent of $W$ and $B^{j,x}=B^j+x$. 
    Here again  the expectation in  (\ref{momentcalculation}) has to be understood as in Remark 1. 
Changing the order of the expectations, yields
    \begin{align} \nonumber
        &\E  \left[ \left(u^{\epsilon}(t,x)\right)^k  \right]\\\nonumber 
        &=\E^{B}\left(\Psi^k_{t,x} \E^W \left( \exp\left(\sum\limits_{j=1}^k \int_0^{t}\int_{ \mathbb{R}^{d}}p_{\epsilon}(B^{j,x}_{t-s}-y)W(ds,dy)-\frac{1}{2}t(2\pi)^d\Lambda_{\epsilon}(0)\right)\right)\right) \\
        \nonumber &= \E \left(\Psi^k_{t,x} \exp\left(\frac{1}{2}\sum\limits_{j,l=1, j\neq l}^k \int_0^{t}\int_{ \mathbb{R}^{2d}}p_{\epsilon}(B^{j,x}_{t-s}-y)p_{\epsilon}(B^{l,x}_{t-s}-y+y')\Lambda(dy')dyds\right)\right)       \\
        \label{ecu4} &= \E \left( \Psi^k_{t,x}\exp\left( \sum\limits_{ 1 \le j< l\le k} \int_0^{t}   \Lambda_{2\e} (B^{j,x}_s - B^{l,x}_s) ds\right)\right).
    \end{align}
   Integrating with respect to the law of the random vector $(B^1_t+x, \dots, B^k_t+x)$ whose density is $\theta \mapsto \prod _{j=1}^k p_t(x-\theta_j)$, the above expectation can be written as follows
    \begin{align*}
 \E \left[ \left(u^{\epsilon}(t,x)\right)^k  \right] &=
    \int_{\R^{kd}}  \prod_{j=1}^k u_0(d\theta_j) p_t(x-\theta_j) \\
    & \quad \times    \E \left( \exp\left( \sum\limits_{ 1 \le j<l\le k} \int_0^{t}   \Lambda_{2\e} 
    (\widehat{B}^{j,x,\theta_j}_{0,t}(s)  - \widehat{B}^{l,x,\theta_l}_{0,t}(s) ) ds\right)\right),
    \end{align*}
    where $\left\{\widehat{B}_{0,t}^{j, \theta_j,x}, j=1,\dots k\right\}$ denote a family of $d$-dimensional Brownian bridges in the interval $[0,t]$ from $x$ to $\theta_j$.
    Now, using the expression (\ref{BB}) for Brownian  bridges, we can write
      \begin{align} \notag
    \E \left[ \left(u^{\epsilon}(t,x)\right)^k  \right] &=
    \int_{\R^{kd}}  \prod_{j=1}^k u_0(d\theta_j) p_t(x-\theta_j) \\
    & \quad \times  \E \left( \exp\left( \sum\limits_{ 1 \le j<l\le k} \int_0^{t}    \Lambda_{2\e} 
    \left(\widehat{B}^{j}_{0,t}(s)  - \widehat{B}^{l}_{0,t}(s)     +\frac { s (\theta_j-\theta_l)}t \right) ds\right)\right).  \label{dn1}
    \end{align}
    
    Now we can proceed with the proof of the proposition. First,  we only need to show (\ref{ecu3}) when $k$ is even. In this case, (\ref{ecu3}) 
    follows from formula  (\ref{dn1}), condition (\ref{initial}) and  (\ref{kappa1}). Indeed,
    we have
    \[
    \E \left[ \left(u^{\epsilon}(t,x)\right)^k  \right] \le  c_t  \left( \int_{\R^{d}}   |u_0|(d\theta) p_t(x-\theta) \right)^k<\infty,
    \]
    where $c_t$ is a finite constant only depending on $t$.

    We claim that  $u^\e(t,x) $ converges in   $L^p(\Omega)$ as $\epsilon \to 0$,
    for all $p\ge 2$.
    Indeed,
    \begin{align*}
    \E\left(u^{\e_1}(t,x)u^{\e_2}(t,x)\right) & = \int_{\R^{2d}} \prod_{j=1}^2 u_0(d\theta_j) p_t(x-\theta_j)\\
    & \quad \times  \E \left( \exp\left( \int_0^t   \Lambda_{\e_1 + \e_2} 
    \left(\widehat{B}^1_{0,t}(s)  - \widehat{B}^{2}_{0,t}(s)     +\frac {s (\theta_1-\theta_2)}t \right) ds\right) \right)
    \end{align*}
converges, as $\epsilon_1 ,\epsilon_2$ tend to $0$, to
\[
\int_{\R^{2d}} \prod_{j=1}^2 u_0(d\theta_j) p_t(x-\theta_j)  \E \left( \exp\left( \int_0^t   \Lambda
    \left(\widehat{B}^1_{0,t}(s)  - \widehat{B}^{2}_{0,t}(s)     +\frac {s (\theta_1-\theta_2)}t \right) ds\right) \right)
    \]
    thanks to Proposition \ref{basic}.
 Therefore, this shows the convergence of $u^{\epsilon}(t,x)$  in   $L^2(\Omega)$ as $\epsilon \to 0$ to  some limit $v(t,x)$. The fact that the convergence is in $L^p(\Omega)$ follows from (\ref{ecu4}) and  Proposition \ref{basic} (i). Taking the limit in (\ref{ecu4}) as $\e$ tends to zero, and using  Proposition \ref{basic} (iii), we obtain the Feynman-Kac formula  (\ref{FeynmanKac}) for the moments of $v(t,x)$.
    
    It remains to show that $v(t,x)$ coincides with the solution to equation (\ref{PAM}).
  By the proof of the Lemma \ref{lem1}, we know that for any random variable of the form $G=e^{W(h)-\frac{1}{2}\|h\|_{\mathcal{H}}^2}$ with $h\in \mathcal{H}$,  
   $u^{\epsilon}$ satisfies
\[
        \E \left(Gu^{\epsilon}(t,x)\right)= (p_t*u_0)(x)+\E \left( \left \langle \int_{[0,t]\times \mathbb{R}^d}p_{t-s}(x-y)u^{\epsilon}(s,y)p_{\epsilon}(x-\star),D_{s,\star}G \right \rangle_{\mathcal{H}_0}\right).
\]
    Now letting $\epsilon \to 0$, we see that
     \begin{align*}
       \E \left(G v(t,x)\right)=(p_t*u_0)(x)+\E \left( \langle vp_{t-\bullet}(x-\star),DG\rangle_{\mathcal{H}}\right),
    \end{align*}
    which implies that the process $v$ 
    is also a solution to the equation (\ref{PAM}), and by uniqueness $v=u$.
   \end{proof}

\section{Proof  Theorem \ref{thm1}}
We have that $u^\e(t,x)$ belongs to $\mathbb{D}^\infty$ for any $(t,x) \in (0,\infty) \times \R^d$. Moreover, we can compute its iterated Malliavin derivative
using the Feynman-Kac formula (\ref{epsilonsolution}):
For any integer  $N \ge 1$ and  for $0<r_1<\cdots<r_N<t$, $z_1,\dots z_N \in \R^d $,
\[
D^N_{\pmb{r}_N, \pmb{z}_N} u^{\epsilon}(t,x)=\E^B \left(u_0(B_t^x)\exp{\left(\int_0^t W^{\epsilon}(ds,B_{t-s}^x)-t(2\pi)^d\Lambda_{\epsilon}(0)\right)}\prod_{m=1}^Np_{\epsilon}(B_{t-r_m}^x-z_m) \right),
\]
where $B$ is a $d$-dimensional Brownian motion independent of $W$ and $B^x_t= B_t+x$.
 Set
\[
M_{k,\e}=\E \left(\left( D^N_{\pmb{r}_N,\pmb{z}_N} u^{\epsilon}(t,x)\right)^k\right)
\]
and \[
M_{k}=\E \left(\left( D^N_{\pmb{r}_N,\pmb{z}_N}u(t,x)\right)^k\right).
\]
After calculations similar to those in (\ref{ecu4}), we get 
\begin{align} \label{momentsofuepsilon}
 M_{k,\e} &= \E  \left( \Psi_{t,x} ^k\exp{\left(\sum_{1\leq j<l\leq k}\int_0^t\Lambda_{2\epsilon}(B_{s}^{j,x}-B_{s}^{l,x})ds \right) } \prod_{j=1}^k \prod_{m=1}^N p_{\epsilon}(B^{j,x}_{t-r_m}-z_m) \right),
\end{align}
where
$\Psi^k_{t,x} =:\prod_{j=1} ^k u_0(B^j_t+x)$ and the $B^j$, $j=1\dots, k$ are independent $d$-dimensional Brownian motions.
The  $k(N+1)d$-dimensional random vector $(B^j_{t-r_N}+x, \dots, B^j_{t-r_1}+x, B^j_t+x)_{ 1\le j \le k}$ has the probability density
\[
 \varphi (\pmb{\eta}, \pmb{\theta})= \prod_{j=1}^k p_{r_1} ( \eta^j_1 -\theta^j)  \left( \prod_{m=1}^{N-1} p_{r_{m+1}-r_m} (\eta^j_{m+1} - \eta^j_m) \right) p_{t-r_N}(x- \eta_N^j),
\]
with the convention  $r_{N+1}=t$, $r_0=0$.
In the above expression, $\eta^j_m$ denotes the value of the random variable $B^j_{t-m}+x$ for $j=1,\dots, k$ and $m=1,\dots, N$ and
$\theta^j$ denotes the value of $B_t^j+x$ for $j=1,\dots,k$ and we make use of the notation   $\pmb{\eta}=(\eta^j_m)_{1\le j \le k ,1\le m \le N} \in\R^{k(N+1)d}$
and $\pmb{\theta}= (\theta^1, \dots, \theta^k)\in \R^{kd}$.

Set  $ \pmb{\eta}^j_N=(\eta_N^j,\dots, \eta_1^j)$ and 
    $\pmb{t}- \pmb{r}_N=(t-r_N,\dots,t-r_1)$.
    Let  
    \begin{equation} \label{bridges}
    \left\{ \widehat{B}_{0, \pmb{t} -\pmb{r}_N, t }^{j,x, \pmb{\eta}^j_N, \theta^j }(s), s\in [0,t]\right\}, \qquad j=1,\dots, k 
    \end{equation}
     be  a  family of $k$ independent $d$-dimensional  Brownian motions,
starting at $x$ and pinned at times $t-r_N<\cdots < t-r_1<t $ to be equal to $\eta^j_N ,\dots, \eta^j_1,\theta^j$, respectively.

Now, conditioning on  $B^j_{t-r_m}+x=\eta_m^j$  and  $B^j_t+x =\theta^j$ for $m=1, \dots, N$ and $j=1,\dots, k$, we can write the expectation in equation (\ref{momentsofuepsilon}) as follows: 
\begin{align}\label{momentsofuepsilonconditional}
M_{k,\e} &=  \int_{\mathbb{R}^{(N+1)kd}} H_{\epsilon}^{t,\pmb{r}_N}(\pmb{\eta},\pmb{\theta})   \left(\prod_{j=1}^{k} \prod_{m=1}^Np_{\epsilon}(\eta_m^j-z_m)\right)
 \varphi (\pmb{\eta}, \pmb{\theta})   u_0(d\pmb{\theta}) d\pmb{\eta} ,
 \end{align}
where
\begin{equation} \label{Hep}
    H_{\epsilon}^{t,\pmb{r}_N}(\pmb{\eta},\pmb{\theta}) :=\E\left(\exp{\left(\sum_{1\leq j<l\leq k}\int_0^t\Lambda_{2\epsilon}(\widehat{B}_{0,\pmb{t}-\pmb{r}_N,t}^{j,x,\pmb{\eta}_N^j,\theta^j}(s)-\widehat{B}_{0,\pmb{t}-\pmb{r}_N,t}^{l,x,\pmb{\eta}_N^l, \theta^l}(s))ds \right)}\right)
    \end{equation}
 and $u_0(d\pmb{\theta})=\prod_{j=1}^{k}u_0(d\theta^j)$.

\begin{proposition} \label{lambdaestimates}
Suppose $\Lambda$ satisfies the Dalang's condition (\ref{Dalangscondition}) and $k\geq 2$ be fixed. With the above notation  let 
$    \{ \widehat{B}_{0, \pmb{t} -\pmb{r}_N, t }^{j,x, \pmb{\eta}^j_N, \theta^j }(s), s\in [0,t]\}$,  $ j=1,\dots, k $
     be  a  family of $k$ independent $d$-dimensional  Brownian motions,
starting at $x$ and pinned at times $t-r_N<\cdots < t-r_1<t $ to be equal to $\eta^j_N ,\dots, \eta^j_1,\theta^j$.
 For $\kappa\in\mathbb{R}$ we set
\[
H^{t,\pmb{r}_N}_{\kappa,\epsilon}(\pmb{\eta},\pmb{\theta})  
    =:\E\left(\exp{\left(\kappa\sum_{1\leq j<l\leq k}\int_0^t\Lambda_{2\epsilon}(\widehat{B}_{0,\pmb{t}-\pmb{r}_N,t}^{j,x,\pmb{\eta}_N^j,\theta^j}(s)-\widehat{B}_{0,\pmb{t}-\pmb{r}_N,t}^{l,x,\pmb{\eta}_N^l, \theta^l}(s))ds \right)}\right).
        \]
    Then, 
\[
    \sup\limits_{\epsilon>0}\,\,\sup\limits_{0<r_1<\cdots<r_N<t}\,\,\sup\limits_{ (\pmb{\eta}, \pmb{\theta})  \in \mathbb{R}^{k(N+1)d}}H^{t,\pmb{r}_N}_{\kappa,\epsilon}(\pmb{\eta},\pmb{\theta})  
        < \infty
\]
and 
$H^{t,\pmb{r}_N}_{\kappa,\epsilon}(\pmb{\eta},\pmb{\theta})$ converges to 
\[
 H_{\kappa}^{t,\pmb{r}_N}(\pmb{\eta},\pmb{\theta})
=: \E\left(\exp{\left(\kappa\sum_{1\leq j<l\leq k}\int_0^t\Lambda\left((\widehat{B}_{0,\pmb{t}-\pmb{r}_N,t}^{j,x,\pmb{\eta}_N^j,\theta^j}(s)-\widehat{B}_{0,\pmb{t}-\pmb{r}_N,t}^{l,x,\pmb{\eta}_N^l, \theta^l}(s)\right)ds \right)}\right),
\]
 as $\epsilon \to 0$ uniformly in $\pmb{\eta}$, $\pmb{\theta}$ and $\pmb{r}_N$.
\end{proposition}

\begin{proof}
Consider the decomposition
\[
    \int_0^t \Lambda_{2\epsilon}\left(\widehat{B}_{0,\pmb{t}-\pmb{r}_N,t}^{j,x,\pmb{\eta}_N^j,\theta^j}(s)-\widehat{B}_{0,\pmb{t}-\pmb{r}_N,t}^{l,x,\pmb{\eta}_N^l, \theta^l}(s))(s) \right)ds=\sum_{m=0}^N\int_{t-r_{m+1}}^{t-r_{m}}\Lambda_{2\epsilon}\left(\widehat{B}_{0,\pmb{t}-\pmb{r}_N,t}^{j,x,\pmb{\eta}_N^j,\theta^j}(s)-\widehat{B}_{0,\pmb{t}-\pmb{r}_N,t}^{l,x,\pmb{\eta}_N^l, \theta^l}(s)\right)ds,
\]
where $r_0=0$ and $r_{N+1}=t$.
For all $j=1,\dots, k$, and  $m=0,1,\dots, N$,  \
\[
\left\{ \widehat{B}_{0,\pmb{t}-\pmb{r}_N,t}^{j,x,\pmb{\eta}_N^j,\theta^j}(s), s\in [t-r_{m+1}, t-r_{m}] \right\} 
\]
 is a Brownian bridge that starts at $\eta^j_{m+1}$ and ends at $\eta^j_{m}$ with the convention $\eta^j_{N+1}=x$ and $\eta^j_0=\theta^j$.    Therefore,  using (\ref{BB}), for each $j=1,\dots, k$, and  $m=0,1,\dots, N$ we can write, for $s\in [t-r_{m+1}, t-r_m]$,
\[
\widehat{B}_{0,\pmb{t}-\pmb{r}_N,t}^{j,x,\pmb{\eta}_N^j,\theta^j}(s) = \widehat{B}^j_{t-r_{m+1}, t-r_{m}}(s)+  \frac {s-t+r_{m+1}}{ r_{m+1} - r_m} \eta^j_m +  \frac {t-r_{m}-s}{ r_{m+1} - r_m} \eta^j_{m+1},
\]
where the  $\widehat{B}^j_{t-r_m,t- r_{m+1}}$ are independent $d$-dimensional Brownian bridges from $0$ to $0$ in each interval $[t-r_{m+1},t- r_{m}]$.  
Moreover, the family of processes
$$\left\{\widehat{B}^j_{t-r_{m+1}, t-r_{m}},  1\le j \le k, 0\le m\leq N\right\} $$ are   independent.
For
 $j,l=1,\dots, k$, and  $m=0,1,\dots, N-1$ set
\[
\alpha^{j,l}_m(s) =\frac {s-r_m}{ r_{m+1} - r_m} (\eta^j_m -\eta^l_m) +  \frac {r_{m+1}-s}{ r_{m+1} - r_m} (\eta^j_{m+1}-\eta^l_{m+1})
\]
and  $\alpha^{j,l}_N= \frac { s-r_
N}{ t-r_N}( \eta_N^j- \eta_N^l )$.
Then,
\[
     H_{\kappa, \epsilon}^{t,\pmb{r}_N}(\pmb{\eta}, \pmb{\theta})= \prod_{m=0}^N F_{\kappa, \epsilon}^{m}(\pmb{\eta}, \pmb{\theta}), 
\]
where  for $m=0, \dots, N$,
\[
    F_{\kappa, \epsilon}^m(\pmb{\eta}, \pmb{\theta}) :=\E\left(\exp{\left(\kappa\sum_{1\leq j<l\leq k}\int_{t-r_{m+1}}^{t-r_{m}}\Lambda_{2\epsilon}\left(B_{t-r_{m+1},t-r_{m}}^{j}(s)-B_{t-r_{m+1},t-r_{m}}^{l}(s)+\alpha_{m}^{jl}(s)\right)ds \right)}\right).
\]
Then by Proposition  \ref{basic} (i) and (iii), we have that for each fixed $t>0$ and $k\geq 2$, 
\begin{align*}
    \sup\limits_{\epsilon>0}\,\,\sup\limits_{0<r_1<\cdots<r_N<t}\,\,\sup\limits_{(\pmb{\eta} , \pmb{\theta}) \in \mathbb{R}^{k(N+1)d}}F^{m}_{\kappa,\epsilon}(\pmb{\eta}, \pmb{\theta}) <\infty,
\end{align*} 
and as $\epsilon \to 0$, $F_{\kappa, \epsilon}^{m}$ converges  uniformly in $\pmb{\eta}$, $\pmb{\theta}$ and $\pmb{r}_N$ to
\[
     F^{m}_\kappa(\pmb{\eta}, \pmb{\theta}):=\E\left(\exp{\left(\kappa\sum_{1\leq j<l\leq k}\int_{t-r_{m+1}}^{t-r_{m}}\Lambda\left(B_{t-r_{m+1},t-r_{m}}^{j}(s)-B_{t-r_{m+1},t-r_{m}}^{l}(s)+\alpha_{m}^{jl}(s)\right)ds \right)}\right).
\]
The proposition follows.
\end{proof}

Proposition \ref{lambdaestimates}  together with the expression for the moments in  (\ref{momentsofuepsilonconditional}) imply that
 the family of random variables $\{D^N_{\pmb{r}_N,\pmb{z}_N}u^{\epsilon}(t,x), \epsilon \in (0,1]\}$ has  uniformly bounded moments of all orders.
 The next result provides the limit as $\e$ tends  to zero of the moment of order $k$ of the iterated derivative of $u^\e(t,x)$.
 
\begin{proposition}{\label{momentsofderivative}}
Let $u^{\epsilon}$ be as defined in (\ref{epsilonsolution}). Then for $k\ge2$, we have, with the notation introduced in Proposition
\ref{lambdaestimates},
\begin{align*}
  &   \lim_{\epsilon \rightarrow 0} \E \left(\left( D_{\pmb{r}_N, \pmb{z}_N}  u^{\epsilon}(t,x)\right)^k\right) 
   = \left[  \prod_{m=1}^{N-1} p_{r_{m+1}-r_m}(z_{m+1}-z_m) \right]^k p^k_{t-r_N}(x-z_N) \\
 & \quad \times   \int_{\R^{kd}}  \prod_{j=1}^k u_0(d\theta^j)  \prod_{j=1}^kp_{r_1}(z_1- \theta^j)  \E  \left(\exp{\left(\sum_{1\leq j<l\leq k}\int_0^t  \Lambda\left(\widehat{B}_{0,\pmb{t-r}_N,t}^{j,x ,\pmb{z}_N, \theta^j}(s)-\widehat{B}_{0,\pmb{t-r}_N,t}^{l,x,\pmb{z}_N, \theta^l}(s)\right)ds \right)}  \right)    .
\end{align*}
\end{proposition}

\begin{proof}  
Notice first that,  by Proposition \ref{basic}, for fixed  $\pmb{z}_N$  the expression
\[
H^{t, \pmb{z}_N}(\pmb{\theta}):=  \E  \left(\exp{\left(\sum_{1\leq j<l\leq k}\int_0^t \Lambda\left(\widehat{B}_{0,\pmb{t-r}_N,t}^{j,x, \pmb{z}_N, \theta^j}(s)-\widehat{B}_{0,\pmb{t-r}_N,t}^{l,x,\pmb{z}_N, \theta^l}(s)\right)ds \right)}  \right)
\]
is a bounded function of the variable   $\pmb{\theta}$. Thus, the integral in the above expression is well defined in view of condition (\ref{initial}).
From (\ref{momentsofuepsilonconditional}), we obtain
\[
M_{k,\e} = M^{(1)}_{k,\e} +  M^{(2)}_{k,\e} ,
\]
where
\[
M^{(1)}_{k,\e} := \int_{\mathbb{R}^{(N+1)kd}} \left(H_{\epsilon}^{t,\pmb{r}_N}(\pmb{\eta},\pmb{\theta})  - H^{t,\pmb{r}_N}(\pmb{\eta},\pmb{\theta})\right) \left( \prod_{j=1}^{k} \prod_{m=1}^Np_{\epsilon}(\eta_m^j-z_m)\right)
 \varphi (\pmb{\eta}, \pmb{\theta})   u_0(d \pmb{\theta}) d\pmb{\eta} 
  \]
    and
\[
M^{(2)}_{k,\e}  :=  \int_{\mathbb{R}^{(N+1)kd}} H^{t,\pmb{r}_N}(\pmb{\eta},\pmb{\theta})   \left(\prod_{j=1}^{k} \prod_{m=1}^Np_{\epsilon}(\eta_m^j-z_m)\right)
 \varphi (\pmb{\eta}, \pmb{\theta})  u_0(d \pmb{\theta}) d\pmb{\eta} ,
\]
with      $H_{\epsilon}^{t,\pmb{r}_N}(\pmb{\eta}, \pmb{\theta} )$  defined in (\ref{Hep}) and
\[
H^{t,\pmb{r}_N}(\pmb{\eta}, \pmb{\theta} ):=E\left(\exp{\left(\sum_{1\leq j<l\leq k}\int_0^t\Lambda_{2\epsilon}\left(\widehat{B}_{0,\pmb{t}-\pmb{r}_N,t}^{j,x,\pmb{\eta}_N^j,\theta^j}(s)-\widehat{B}_{0,\pmb{t}-\pmb{r}_N,t}^{l,x,\pmb{\eta}_N^l, \theta^l}(s)\right)ds \right)}\right).
\]

\medskip
\noindent
{\it Convergence of $M^{(2)}_{k,\e} $}: We claim that
\begin{align*}
\lim_{\epsilon \rightarrow 0} M^{(2)}_{k,\e}
& =\left[  \prod_{m=1}^{N-1} p_{r_{m+1}-r_m}(z_{m+1}-z_m) \right]^k p^k_{t-r_N}(x-z_N)\\
& \qquad \times  \int_{\R^{kd}}     \left(\prod_{j=1}^kp_{r_1}(z_1- \theta^j)\right)H^{t, \pmb{z}_N}(\pmb{\theta}) u_0(d \pmb{\theta})   .
\end{align*}
Indeed, $    M^{(2)}_{k,\e}= 
    \left( G * \phi_{\epsilon}\right)(\pmb{z}_N, \dots,  \pmb{z}_N),
$
where
\[  
 G(\pmb{\eta})= \int_{\R^{kd}}   H^{t,\pmb{r}_N}(\pmb{\eta}, \pmb{\theta} ) \left( \prod_{j=1}^k
 p_{r_1}( \eta_1^j -\theta^j)\right) \left(
 \prod_{m=1}^N p_{r_{m+1}-r_m}(\eta_{m+1}^j-\eta_{m}^j)  \right)u_0(d \pmb{\theta}) 
\]
with $\eta^j_{N+1}=x$
 and
 \[
\phi_{\epsilon}(\pmb{\eta})=\prod_{j=1}^k\prod_{m=1}^Np_{\epsilon}(\eta_m^j).
\]
Notice first  that $G$ is  integrable.  In fact, taking into account that  $H^{t,\pmb{r}_N}(\pmb{\eta}, \pmb{\theta} )$ is uniformly bounded by Proposition \ref{lambdaestimates}, we can write
\begin{align*}
  \int_{\R^{kNd}} |G(\pmb{\eta})| d \pmb{\eta}&  \le C   \int_{\R^{k(N+1) d}}  \prod_{j=1}^k
 p_{r_1}( \eta_1^j -\theta^j) \left(
 \prod_{m=1}^N p_{r_{m+1}-r_m}(\eta_{m+1}^j-\eta_{m}^j)  \right) |u_0|(d \pmb{\theta}) d\pmb{\eta} \\
&\le C    \left(\int_{\R^{ d}}  |u_0|(d\theta)  
 p_{t}( x -\theta) \right)^k <\infty.
\end{align*}
Moreover, again by Proposition \ref{lambdaestimates}, one can show that $G$ is uniformly continuous in $\pmb{\eta}_N^1,\dots,\pmb{\eta}^k_N$. Taking into account that  $\phi_{\epsilon}$ is an approximation to identity in $\R^{Nkd}$, we obtain  that\\ $\left( G * \phi_{\epsilon}\right)(\pmb{z}_N,\dots  \pmb{z}_N)$ converges to $ G(\pmb{z}_N,\dots,  \pmb{z}_N)$ as $\epsilon \to 0$. So, it  only remains to show that the first  term  $M^{(1)}_{k,\e}$ converges to $0$ as $\epsilon$ goes to $0$. This follows from Proposition \ref{lambdaestimates} combined with the dominated convergence theorem.
\end{proof}

In the next proposition we show that the approximated Malliavin derivatives converge in $L^p(\Omega)$, for each $p\ge 2$.

\begin{proposition}\label{Phi} For all  $0<r_1, \dots, r_N<t$ and $z_1, \dots, z_N \in \R^d$, 
 \[\displaystyle
 \lim_{\epsilon \to 0}D^N_{\pmb{r}_N, \pmb{z}_N} u^{\epsilon}({t,x})
  \] 
  exists in $L^p(\Omega)$, for all $p\ge 2$. We denote this limit by $\Phi_{\pmb{r}_N,\pmb{z}_N}(t,x)$.
  \end{proposition}
\begin{proof}
First we show  that $D_{\pmb{r}_N, \pmb{z}_N}u^{\epsilon}({t,x})$ converges in $L^2(\Omega)$ as $\e$ tends to zero. For $\epsilon_1,\epsilon_2>0$, by  calculations similar to (\ref{momentsofuepsilon}) and denoting by
 $B^1$ and $B^2$ two  independent $d$-dimensional Brownian motions, we have
\begin{align*}
   &\E\Big(D^N_{\pmb{r}_N, \pmb{z}_N}u^{\epsilon_1}({t,x})D^N_{\pmb{r}_N, \pmb{z}_N}u^{\epsilon_2}({t,x}) \Big)\\ &= \E\left( \exp\left( \int_{0}^t \Lambda_{\epsilon_1+\epsilon_2} (B^1_{s}-B^2_{s})ds\right)\prod_{m=1}^N p_{\epsilon_1}(B_{t-r_m}^{1,x}-z_m)p_{\epsilon_2}(B_{t-r_m}^{2,x}-z_m)\right) \\
   &= \int_{\R^{2(N+1)d}}
   u_0(d\theta^1) u_0(d\theta^2) 
    \E\left( \exp\left( \int_{0}^t \Lambda_{\epsilon_1+\epsilon_2} \left(\widehat{B}^{1,x,\pmb{\eta}^1_N, \theta^1}_{0,\pmb{t-r}_N,t}{(s)}-\widehat{B}^{2,x,\pmb{\eta}^2_N,\theta^2}_{0,\pmb{t-r}_N},t{(s)}\right)ds\right)\right)\\ 
   &  \hspace{2cm}  \times 
   \prod_{j=1}^2     \prod_{m=1}^N p_{\epsilon_j}(\eta^j_m-z_m) \varphi(\pmb{\eta}, \pmb{\theta}) d\pmb{\eta} .
 \end{align*}
 Following the same proof as in (\ref{momentsofderivative})  we obtain that $\E\Big(D_{\pmb{r}_N, \pmb{z}_N}u^{\epsilon_1}({t,x})D_{\pmb{r}_N, \pmb{z}_N}u^{\epsilon_2}({t,x}) \Big)$ converges, as $\epsilon_1,\epsilon_2 \to 0$, to
 \begin{align*}  
& \left[  \prod_{m=1}^{N-1} p_{r_{m+1}-r_m}(z_{m+1}-z_m) \right]^2 p^2_{t-r_N}(x-z_N) \\
 & \quad \times   \int_{\R^{2d}}    u_0(d\theta^1)u_0(d\theta^2)   \prod_{j=1}^2 p_{r_1}(z_1- \theta^j)  \E  \left(\exp{\left( \int_0^t  \Lambda(\widehat{B}_{0,\pmb{t-r}_N,t}^{1,x ,\pmb{z}_N, \theta^1}(s)-\widehat{B}_{0,\pmb{t-r}_N,t}^{2,x,\pmb{z}_N, \theta^2}(s))ds \right)}  \right)    .
 \end{align*} 
This implies the convergence in $L^2(\Omega)$. By the boundedness of all moments of all orders, the convergence is in $L^p(\Omega)$ for all $p\ge 2$.
\end{proof}

The next step in the proof of Theorem \ref{thm1} is to show that  the limit  random field $\Phi_{\pmb{r}_N,\pmb{z}_N}(t,x)$ appearing in Proposition \ref{Phi} is precisely the iterated derivative  $D^N_{\pmb{r}_N, \pmb{z}_N} u({t,x})$. We recall that the iterated derivative $D^N u({t,x})$ is an $\mathcal{H}^{\otimes N}$-valued random variable. In the next result we will show first that $D^N u^{\e} ({t,x})$ converges to $D^N u({t,x})$ in  $L^2(\Omega;\mathcal{H}^{\otimes N})$, as $N$ tends to infinity. The proof of this fact will be based on the Wiener chaos expansion of  the solution $u(t,x)$ that we recall here.
The following chaos expansion was shown in \cite{hu2015stochastic}\footnote{In \cite{hu2015stochastic} the initial condition is continuous and bounded, but the result still holds for initial conditions satisfying (\ref{initial}).},
\[
    u(t,x)=(p_t* u_0)(x)+\sum_{n=1}^{\infty}I_n(f_{n,t,x}),
    \]
    where  $I_n$ denotes the multiple stochastic integral with respect to the noise $W$ and  $f_{n,t,x}\in \mathcal{H}^{\otimes n}$ is the symmetric kernel given by
   \[
    f_{n,t,x}(s,y):=f_{n,t,x}(s_1,y_1,\dots,s_n,y_n)=\frac{1}{n!} (p_{s_{\sigma(1)}}*u_0)(x_{\sigma(1)})\prod_{i=1}^n p_{s_{\sigma(i+1)}-s_{\sigma(i)}}(y_{\sigma{(i+1)}}-y_{\sigma{(i)}})
\]
for $(s_1,y_1,\dots,s_n,y_n) \in ((0,t) \times \R^d)^n$ satisfying $s_i \not= s_j$ for $i \not =j$.
In this expression,  $\sigma$ denotes the permutation of $\{1,\dots,n\}$ such that $0<s_{\sigma(1)}<\cdots <s_{\sigma(n)}<t$ and we use the convention 
$\sigma(n+1)=t$ and $y_{\sigma(n+1)}=x$.
In the same way, we can derive the Wiener chaos expansion of  $u^{\epsilon}(t,x)$ with respect to the noise $W$:
\[
u^{\epsilon}(t,x)= (p_t* u_0)(x)+ \sum_{n=1}^{\infty}I_n(f^{\epsilon}_{n,t,x}), 
\]
where
\begin{align*}
    f^{\epsilon}_{n,t,x}(s,y)&:=f^{\epsilon}_{n,t,x}(s_1,y_1,\dots,s_n,y_n)=\frac{1}{n!}\\
    & \quad \times \int_{\mathbb{R}^{nd}}
    (p_{s_{\sigma(1)}}*u_0)(w_1)\prod_{i=1}^n p_{s_{\sigma(i+1)}-s_{\sigma(i)}}(w_{i+1}-w_{i})p_{\epsilon}(y_{\sigma(i)}-w_{i})dw_{1} \cdots dw_n,
\end{align*}
with the same convention about $\sigma$ mentioned above.

\begin{proposition} \label{H^N} $\displaystyle
 \lim_{\epsilon \to 0}D^Nu^{\epsilon}({t,x}) = D^Nu({t,x})$ in $L^2(\Omega;\mathcal{H}^{\otimes N})$.
 \end{proposition}

\begin{proof}
Define
 \[
  g_{n,t,x}^{\epsilon}(s,y):=f_{n,t,x}^{\epsilon}(s,y)-f_{n,t,x}(s,y).
\]
 We know that, in terms of the Wiener chaos expansion, the Malliavin derivative is obtained by leaving one variable free and multiplying by the order of the chaos. That is, 
 \[
 D^N_{\pmb{r}_N, \pmb{z}_N}  \left(u^{\epsilon}({t,x})-u({t,x})\right)=\sum_{n=N}^{\infty}n(n-1)\dots (n-N+1) I_{n-N}((g^{\epsilon}_{n,t,x}(\bullet,r_1,z_1,\dots,r_N,z_N))),
 \]
 for any  $z_1, \dots, z_N \in \R^d$ and $0<r_1 < \cdots < r_N$ and with the notation $\pmb{r}_N=(r_1, \dots, r_N)$ and $\pmb{z}_N=(z_1, \dots, z_N)$.
 This leads to
 \[
 \E\Big( \left| D^N _{\pmb{r}_N, \pmb{z}_N}  \left(u^{\epsilon}({t,x})-u({t,x})\right)\right|^2\Big)
 \le  \sum_{n=N}^{\infty}  n^N n! \|  g^{\epsilon}_{n,t,x}(\bullet,r_1,z_1,\dots,r_N,z_N) \|^2_{\mathcal{H}^{\otimes (n-N)}},
 \]
 which implies
  \[
 \E\left( \left \| D^N    \left(u^{\epsilon}({t,x})-u({t,x})\right)\right\|_{\mathcal{H}^{\otimes N}}^2\right)
 \le  \sum_{n=N}^{\infty}  n^N n! \|  g^{\epsilon}_{n,t,x}  \|^2_{\mathcal{H}^{\otimes n}}.
 \]
 Using the notation $\pmb{\xi} =(\xi_1, \dots, \xi_n)$,  $\mu(d \pmb{\xi})=\prod_{i=1}^n \mu(d\xi_1)$ and $d \pmb{s}= \prod_{i=1} ^n ds_i$, we have
\[
     \|g^{\epsilon}_{n,t,x}\|^2_{\mathcal{H}^{\otimes n}} = \int_{[0,t]^n}\int_{\mathbb{R}^{nd}}\left|\mathscr{F}g^{\epsilon}_{n,t,x}(\xi)\right|^2\mu(d\pmb{\xi})d\pmb{s},
\]
where $\mathscr{F}$ denotes the Fourier transform in the space variables.
 Now let us calculate the Fourier transform of $g_{n,t,x}^{\epsilon}$:
 \[
   \left|\mathscr{F}g^{\epsilon}_{n,t,x}(\pmb{\xi})\right|^2=\frac{1}{(n!)^2}\left(1-e^{-\epsilon \sum_{i=1}^n |\xi_i|^2}\right)\prod_{i=1}^n \exp\left(-(s_{\sigma(i+1)}-s_{\sigma(i)})\left|\sum_{j=1}^i\xi_{\sigma(j)}\right|^2\right),
   \]
  which tends to zero for all $\pmb{\xi}\in \R^{nd}$ and for all $n\ge 1$, as $\e\to 0$.
  
 By the dominated convergence theorem, in order to show   that $\lim_{\epsilon \to 0}D ^N u^{\epsilon}({t,x}) = D ^N u({t,x})$
   in $L^2(\Omega; \mathcal{H}^{\otimes N})$, 
  it suffices to check that
\[
I:= \sum_{n=1}^{\infty}  \frac {n^{N}}{n!} \int_{[0,t]^n}\int_{\mathbb{R}^{nd}}   \prod_{i=1}^n \exp\left(-(s_{\sigma(i+1)}-s_{\sigma(i)})\left|\sum_{j=1}^i\xi_{\sigma(j)}\right|^2\right) \mu(d\pmb{\xi})d\pmb{s}  <\infty.
 \]
 We can write
 \begin{align*}
 I&=\sum_{n=1}^{\infty}  n^{N} \int_{[0,t]_<^n}\int_{\mathbb{R}^{nd}}   \prod_{i=1}^n \exp\left(-(s_{i+1}-s_i )\left| \xi_1 + \cdots \xi_j \right|^2\right) \mu(d\pmb{\xi})d\pmb{s}, 
 \end{align*}
    where we used the notation $[0,t]_<^n:=\{(s_1,\dots,s_n):0<s_1<\cdots<s_n<t\}$ and the convention $s_0=0$.  This leads to the estimate
      \begin{align} \notag
   I   &\leq  \sum_{n=1}^{\infty}  n^{N} \int_{[0,t]^n_<} \prod_{j=1}^n \left(\sup_{\eta \in \mathbb{R^d}} \int_{\mathbb{R}^d} \exp{\left( - (s_j-s_{j-1})|\xi_j+\eta|^2 \right)}d\mu(\xi_j)\right) ds_1 \cdots ds_n \\   \label{BB1}
      &= \sum_{n=1}^{\infty}  n^{N} \int_{[0,t]^n_<} \prod_{j=1}^n \left( \int_{\mathbb{R}^d} \exp{\left( - (s_j-s_{j-1})|\xi_j|^2 \right)}d\mu(\xi_j)\right) ds_1 \cdots ds_n=:  \sum_{n=1}^{\infty}  n^{N} J_n,
      \end{align}
   because, as it is easy to check using the spectral measure $\mu$, the above supremum is attained at $\eta=0$.
     Let $w_j:=s_j-s_{j-1}$, $d\pmb{w}=dw_1\cdots dw_n$, and $A_{t,n} :=\{(w_1,\dots,w_n): w_j \geq 0, j=1,\dots, n \mbox{ and } w_1+\cdots+w_n \leq t\}$.  Then,   integral in (\ref{BB1} is equal to 
     \begin{align*}
        J_n= \int_{A_{t,n}} \int_{\mathbb{R}^{dn}} \exp{\left( -\sum_{j=1}^n w_j|\xi_j|^2 \right)}d\mu(\pmb{\xi}) d\pmb{w}.
    \end{align*}
    To estimate this quantity, set
     \begin{align*}
        C_M:= \int_{|\xi| \geq M}\frac{\mu(d\xi)}{|\xi|^2}, \mbox{   and     } D_M:=\mu(\{\xi \in \mathbb{R}^d: |\xi|\leq M\}).
    \end{align*}
   By Dalang's condition (\ref{Dalangscondition}), both $C_M,D_M$ are finite, and, we can choose $M >0$ such that $ C_M <\frac 14$.
Thus, by  \cite[Lemma 3.3]{hu2015stochastic}, we get
 \begin{equation}\label{BB2}
      J_n\leq \sum_{k=0}^n \binom{n}{k}\frac{t^k}{k!}D_M^k (2C_M)^{n-k}.
    \end{equation}
Substituting (\ref{BB2}) into  (\ref{BB1}),     we get
\[
   I   \leq \sum_{n=1}^{\infty} \sum_{k=0}^n n^N\binom{n}{k}\frac{t^k}{k!}D_M^k (2C_M)^{n-k} 
  \leq \sum_{k=0}^{\infty} \frac{(2D_Mt)^k}{k!} \sum_{n=k}^{\infty} n^N (4C_M)^{n-k} <\infty,
\]
which allows us to conclude the proof.
\end{proof}

In the next result we show that  $D ^N u({t,x})$ is actually a function and it coincides with $(\pmb{r}_N,\pmb{z}_N) \mapsto  \Phi_{\pmb{r}_N,\pmb{z}_N}(t,x)$,
where $\Phi_{\pmb{r}_N,\pmb{z}_N}(t,x)$ is the limiting random field in  Proposition \ref{Phi}.

\begin{proposition}  \label{prop4.5} Let $\Phi_{\pmb{r}_N,\pmb{z}_N}(t,x)$ be as in Proposition \ref{Phi}. Then, $\Phi_{\pmb{r}_N,\pmb{z}_N}(t,x)=D^N_{\pmb{r}_N,\pmb{z}_N}u(t,x)$ for almost all $(\pmb{r}_N, \pmb{z}_N) \in  [0,t]^N_< \times \R^{Nd}$.
\end{proposition}

\begin{proof} The proof will be done in two steps.

\medskip
\noindent
  {\it Step 1:} In this step we will show that  for any bounded and continuous function $\Psi: [0,t]^N \times \R^{N d} \rightarrow \R$,  the limit 
 \begin{equation} \label{conv1}
 \lim_{\epsilon \to 0}  \langle D^N  u^{\epsilon}({t,x}) ,\Psi \rangle_{L^2([0,t]^N\times \R^{Nd})} =  \langle \Phi  (t,x), \Psi \rangle_{L^2([0,t]^N\times \R^{Nd})}
  \end{equation}
   holds in $L^2(\Omega) $. We already know, by Proposition  \ref{Phi}, that for all $0<r_1 < \cdots <r_N <t$ and for all  $ \pmb{z}_N\in  \R^{dN}$
  the limit   $\lim_{\epsilon \to 0}D^N_{\pmb{r}_N, \pmb{z}_   N}  u^{\epsilon}({t,x}) = \Phi_{\pmb{r}_N,\pmb{z}_
   N}  (t,x)$ holds in $L^2(\Omega)$.  Then,  the convergence (\ref{conv1}) holds if
   \begin{align*}
 &  \lim_{\epsilon_1, \e_2 \to 0}  \E  \left( \int _{([0,t]^N_<)^2}\int_{\R^{2Nd}} D^N_{\pmb{r}_N, \pmb{z}_N}  u^{\epsilon_1}({t,x})D^N_{\pmb{s}_N, \pmb{y}_N}  u^{\e_2}({t,x})  \Psi( \pmb{r}_N,\pmb{z}_N) \Psi( \pmb{s}_N,\pmb{y}_N)
   d\pmb{z}_N    d\pmb{y}_N  d \pmb{r}_N d \pmb{s}_N \right)\\
   &\qquad 
  =   \E \left(  \int _{([0,t]^N_<)^2} \int_{\R^{2Nd}} \Phi_{\pmb{r}_N, \pmb{z}_N }  (t,x)  
\Phi_{\pmb{s}_N, \pmb{y}_N }  (t,x)  \Psi( \pmb{r}_N,\pmb{z}_N) \Psi( \pmb{s}_N,\pmb{y}_N)
   d\pmb{z}_N    d\pmb{y}_N   d \pmb{r}_N d \pmb{s}_N  \right).
   \end{align*}
   Set 
   \[
   M_{\e_1,\e_2}=  \E  \left( \int _{([0,t]^N_<)^2}\int_{\R^{2Nd}} D^N_{\pmb{r}_N, \pmb{z}_N}  u^{\epsilon_1}({t,x})D^N_{\pmb{s}_N, \pmb{y}_N}  u^{\e_2}({t,x})  \Psi( \pmb{r}_N,\pmb{z}_N) \Psi( \pmb{s}_N,\pmb{y}_N)
   d\pmb{z}_N    d\pmb{y}_N  d \pmb{r}_N d \pmb{s}_N \right).
   \]
  We have
  \begin{align*}
  M_{\e_1,\e_2} &= \int _{([0,t]^N_<)^2} d \pmb{r}_N  d \pmb{s}_N  \int_{\R^{2Nd}}  d\pmb{z}_N    d\pmb{y}_N
   \Psi( \pmb{r}_N,\pmb{z}_N) \Psi( \pmb{s}_N,\pmb{y}_N) \\
   & \qquad \times
   \int_{\R^{2(N+1)d}}   H^{2, \pmb{r}_N}_{\e_1,\e_2}   (\pmb{\eta}, \pmb{\theta})
   u_0(d\theta^1) u_0(d\theta^2)   
 \left( \prod_{m=1}^N p_{\epsilon_1}(\eta^1_m-z_m)p_{\epsilon_2}(\eta^2_m-y_m)\right) \varphi(\pmb{\eta}, \pmb{\theta}) d\pmb{\eta} ,
  \end{align*}
  where
  \[
  H^{2, \pmb{r}_N}_{\e_1,\e_2}   (\pmb{\eta}, \pmb{\theta})=:
  \E\left( \exp\left( \int_{0}^t \Lambda_{\epsilon_1+\epsilon_2} \left(\widehat{B}^{1,x,\pmb{\eta}^1_N, \theta^1}_{0,\pmb{t-r}_N,t}{(s)}-\widehat{B}^{2,x,\pmb{\eta}^2_N,\theta^2}_{0,\pmb{t-r}_N,t}{(s)}\right)ds\right)\right),
  \]
  where we used the  notation introduced in  (\ref{bridges}).
  From Proposition \ref{basic}, we know that  $H^{2, \pmb{r}_N}_{\e_1,\e_2}   (\pmb{\eta}, \pmb{\theta})$  is uniformly bounded and  converges as 
  $\e_1$ and $\e_2$ tend to zero, uniformly in $\pmb{\eta}$, $\pmb{\theta}$ and $\pmb{r}_N$  to
    \[
  H^{2, \pmb{r}_N}   (\pmb{\eta}, \pmb{\theta})=:
  \E\left( \exp\left( \int_{0}^t \Lambda \left(\widehat{B}^{1,x,\pmb{\eta}^1_N, \theta^1}_{0,\pmb{t-r}_N,t}{(s)}-\widehat{B}^{2,x,\pmb{\eta}^2_N,\theta^2}_{0,\pmb{t-r}_N,t}{(s)}\right)ds\right)\right).
  \]
  Moreover,
  \begin{align*}
&\int _{([0,t]^N_<)^2} d \pmb{r}_N  d \pmb{s}_N  \int_{\R^{2Nd}}  d\pmb{z}_N    d\pmb{y}_N
 |  \Psi( \pmb{r}_N,\pmb{z}_N) \Psi( \pmb{s}_N,\pmb{y}_N)| \\
   & \qquad \times
   \int_{\R^{2(N+1)d}}    
   |u_0|(d\theta^1) |u_0|(d\theta^2)   
  \left( \prod_{m=1}^N p_{\epsilon_1}(\eta^1_m-z_m)p_{\epsilon_2}(\eta^2_m-y_m)\right) \varphi(\pmb{\eta}, \pmb{\theta}) d\pmb{\eta}\\
  & \le \| \Psi\|_\infty^2  \frac {t^N}{N!} \left( \int_{\R^d} |u_0| (dy) p_t(x-y) \right)^2 <\infty.
\end{align*}
   As a consequence, 
   \[
      \lim_{\epsilon_1, \e_2 \to 0}    M_{\e_1,\e_2} =   \lim_{\epsilon_1, \e_2 \to 0}    N_{\e_1,\e_2},
      \]
      where
        \begin{align*}
  N_{\e_1,\e_2} &= \int _{([0,t]^N_<)^2} d \pmb{r}_N  d \pmb{s}_N  \int_{\R^{2Nd}}  d\pmb{z}_N    d\pmb{y}_N
   \Psi( \pmb{r}_N,\pmb{z}_N) \Psi( \pmb{s}_N,\pmb{y}_N) \\
   & \qquad \times
   \int_{\R^{2(N+1)d}}   H^{2, \pmb{r}_N}   (\pmb{\eta}, \pmb{\theta})
   u_0(d\theta^1) u_0(d\theta^2) \left(  
   \prod_{m=1}^N p_{\epsilon_1}(\eta^1_m-z_m)p_{\epsilon_2}(\eta^2_m-y_m)\right) \varphi(\pmb{\eta}, \pmb{\theta}) d\pmb{\eta} .
  \end{align*}
   Finally, taking into account that $\Psi$ is continuous and bounded, 
   we deduce the convergence
   \begin{align*}
   \lim_{\epsilon_1, \e_2 \to 0}    N_{\e_1,\e_2} &=
\int _{([0,t]^N_<)^2} d \pmb{r}_N  d \pmb{s}_N  \int_{\R^{2Nd}}  d\pmb{z}_N    d\pmb{y}_N
   \Psi( \pmb{r}_N,\pmb{z}_N) \Psi( \pmb{s}_N,\pmb{y}_N) \\
   & \qquad \times
   \int_{\R^{2(N+1)d}}   H^{2, \pmb{r}_N}   (\pmb{z}_N,  \pmb{y}_N, \pmb{\theta})
   u_0(d\theta^1) u_0(d\theta^2)   
    \varphi(\pmb{\eta}, \pmb{\theta}) d\pmb{\eta} ,
  \end{align*}
  where
  \[
  H^{2, \pmb{r}_N}   (\pmb{z}_N, \pmb{y}_N, \pmb{\theta})=:
  \E\left( \exp\left( \int_{0}^t \Lambda \left(\widehat{B}^{1,x,\pmb{z}_N, \theta^1}_{0,\pmb{t-r}_N,t}{(s)}-\widehat{B}^{2,x,\pmb{y}_N,\theta^2}_{0,\pmb{t-r}_N,t}{(s)}\right)ds\right)\right).
  \]
  This shows that   
  \[
  \lim_{\epsilon_1, \e_2 \to 0}    M_{\e_1,\e_2}=
 \E \left(  \int _{([0,t]^N_<)^2} \int_{\R^{2Nd}} \Phi_{\pmb{r}_N, \pmb{z}_N }  (t,x)  
\Phi_{\pmb{s}_N, \pmb{y}_N }  (t,x)  \Psi( \pmb{r}_N,\pmb{z}_N) \Psi( \pmb{s}_N,\pmb{y}_N)
   d\pmb{z}_N    d\pmb{y}_N   d \pmb{r}_N d \pmb{s}_N  \right)
   \] 
   and completes the proof of the convergence (\ref{conv1}).

  \medskip
 \noindent
  {\it Step 2:}
  We want to show that $D^Nu(t,x)$, which is defined as a random variable taking values in $\mathcal{H}^{\otimes N}$,  is actually a function and coincides with $\Phi^N(t,x)$, for almost  all $(\pmb{r}_N, \pmb{z}_N) \in  [0,t]^N_< \times \R^{Nd}$, almost everywhere. Because $\mathscr{S}(\R_+^N \times \R^{Nd})$ is dense in $\mathcal{H}^{\otimes N}$, it suffices to show that
  for every  $h \in \mathscr{S}(\R_+^N \times \R^{Nd})$, we have
  \begin{equation} \label{eq2}
  \langle  D^Nu(t,x), h \rangle_{\mathcal{H}^{\otimes N}} =  \langle  \Phi^N(t,x), h\rangle_{\mathcal{H}^{\otimes N}}.
  \end{equation}
  We know, from Proposition \ref{H^N} above, that  $  \langle  D^Nu(t,x), h\rangle_{\mathcal{H}^{\otimes N}}$ is the limit in  $L^2(\Omega)$ as $\e \to 0$ of  $  \langle  D^Nu^\e(t,x), h\rangle_{\mathcal{H}^{\otimes N}}$. Moreover,
\[
  \langle  D^Nu^\e(t,x), h \rangle_{\mathcal{H}^{\otimes N}}
  = \int_{[0,t]_<^N}  \int_{\R^{2Nd}}  D^N_{\pmb{r}_N, \pmb {z}_N}u^\e(t,x)   h(\pmb{r}_N, \pmb{z}_N-\pmb{y}_N)  \Lambda(d\pmb{y}_N) d\pmb{z}_N  
  d \pmb{r}_N.
  \]
  The function
  \begin{equation} \label{Psi}
  \Psi (\pmb{r}_N, \pmb{z}_N) := \int_{\R^{Nd} } h(\pmb{r}_N, \pmb{z}_N-\pmb{y}_N) \Lambda(d\pmb{y}_N) 
  \end{equation}
  is continuous and bounded because it can be written as
  \[
  \Psi (\pmb{r}_N, \pmb{z}_N) = \int_{\R^{Nd} }  e^{-i  \pmb{z}_N \cdot \pmb{\xi}} \mathscr{F}h(\pmb{r}_N,  \pmb{\xi} )  \mu(d\pmb{\xi})
  \]
  and $  \int_{\R^{Nd} }   | \mathscr{F}h(\pmb{r}_N,  \pmb{\xi} )  |\mu(d\pmb{\xi}) < \infty$.
  From  Step 1 applied to the function $\Psi$ defined in (\ref{Psi}),  we deduce
\begin{align*}
   \lim _{\e \to 0}  \langle  D^Nu^\e(t,x), h \rangle_{\mathcal{H}^{\otimes N}}
  &= \int_{[0,t]_<^N}  \int_{\R^{2Nd}}   \Phi^N _{\pmb{r}_N, \pmb{r}_N}  h(\pmb{r}_N, \pmb{z}_N-\pmb{y}_N) \Lambda(d\pmb{y}_N) d\pmb{z}_N  
  d \pmb{r}_N\\
  &=\langle  \Phi^N(t,x), h\rangle_{\mathcal{H}^{\otimes N}},
  \end{align*}
 where the convergence is in $L^2(\Omega)$. This completes the proof of the proposition.
  \end{proof}

\begin{proof}[Proof of Theorem \ref{thm1}]
From Proposition \ref{prop4.5},  iterated Malliavin derivative $D^N_{\pmb{r}_N,\pmb{z}_N}u(t,x)$ coincides with the random field
$\Phi_{\pmb{r}_N,\pmb{z}_N}(t,x)$ for almost all $\pmb{r}_N$ and $\pmb{z}_N$.
Therefore, using Proposition   \ref{Phi}   the moment of order $k$  of $D^N_{\pmb{r}_N,\pmb{z}_N}u(t,x)$ will be the limit as $\e \to 0$ of the moment of order $k$ of $D^N_{\pmb{r}_N,\pmb{z}_N}u^\e(t,x)$, which has been computed in Proposition \ref{momentsofderivative}. 
\end{proof}

\end{document}